\documentclass[11pt, a4paper, oneside]{article} 
\usepackage[textsize=footnotesize,textwidth=3cm]{todonotes}

\newcommand{\Lam}{\Lambda}
\newcommand{\Lamc}{{\bar{\Lambda}}}
\usepackage{amsthm}
\newtheorem{thm}{Theorem}
\newtheorem{theorem}[thm]{Theorem}

\newtheorem{corollary}[thm]{Corollary}

\theoremstyle{definition}

\theoremstyle{remark}

\PassOptionsToPackage{override}{xcolor} 
\usepackage[utf8]{inputenc}
\usepackage{hyperref}
\usepackage{subfig}
\usepackage{amsmath} 
\usepackage{amssymb}
\usepackage{amsfonts,amssymb}
\usepackage{dsfont}
\usepackage[mathscr]{euscript}
\usepackage{enumerate,xspace,threeparttable}
\usepackage{graphicx}
\usepackage{verbatim}
\usepackage{algorithmic}
\usepackage{listings}
\usepackage{wrapfig}
\usepackage{translator}

\providecommand{\qr}{\eqref}
\renewcommand{\d}{\,d}

\providecommand{\RR}{\mathbb{R}}

\providecommand{\ZZ}{\mathbb{Z}}

\providecommand{\CF}{\mathscr{F}}
\providecommand{\CB}{\mathscr{B}}

\providecommand{\CM}{\mathscr{M}}

\providecommand{\mc}{\mathcal}

\providecommand{\opn}{\operatorname}
\providecommand{\ol}{\overline}

\renewcommand{\P}{\mathsf{P}}

\providecommand{\var}{\opn{var}}

\providecommand{\Ordo}[1]{{O(#1)}}

\providecommand{\tl}{\tilde}

\providecommand{\T}{^{\!\mathrm{T}}}

\def\X{\mscr X}
\def\T{\opn{\msf{T}}}

\def\P{\opn{P}}

\title{Phase transitions in long-range Ising models and an optimal condition for factors of $g$-measures}
\author{Anders Johansson, Anders \"Oberg, and Mark Pollicott}
\date{}
\begin{document}
\maketitle
\begin{abstract}
  We weaken the assumption of summable variations in a paper by
  Verbitskiy \cite{verb} to a weaker condition, Berbee's condition, in order for a
  1-block factor (a single site renormalisation) of the full shift
  space on finitely many symbols to have a $g$-measure with a
  continuous $g$-function. But we also
  prove by means of a counterexample, that this condition is (within
  constants) optimal. The counterexample is based on the second of our main
  results, where we prove that there is an inverse
  critical temperature in a one-sided long-range Ising model which is
  at most 8 times the critical inverse temperature for the (two-sided) Ising
  model with long-range interactions.
\end{abstract} {\small \tableofcontents}
\section{Introduction}\noindent
The study of $g$-measures has a long history, with notable
achievements in the 1930's in Romania and France under the name {\em
  chains with complete connections} and which refers to a
generalization of Markov chains (and more generally, chains with
finite memory) on a finite set to chains that have infinite
memory. The $g$-measures are the not necessarily unique stationary distributions 
for such chains. We use the terminology of $g$-measures (with respect to a 
continuous transition probability function $g$) that was introduced by Keane in his
1972 paper \cite{keane} and used in other papers of ours related to this investigation
(\cite{johob}, \cite{jop1}, \cite{jop2}).

Consider a left full shift map $T$ on infinite strings of finitely
many symbols, $X=S^{{\mathbb Z}_+}$, i.e., $S$ is a finite set. Thus
$T$ acts on elements $x$ of $X$, $x=(x_0,x_1,x_2,\ldots)$, in the
following way (each $x_i$ belongs to $S$):
$$T(x_0,x_1,x_2,\ldots)=(x_1,x_2,\ldots).$$
A $g$-measure is a $T$-invariant Borel probability measure $\mu$ that
is associated with a continuous function $g:X\to [0,1]$ so that
$g=d\mu/(d\mu \circ T)$ with $\sum_{y\in T^{-1}x}g(y)=1$, for all
$x\in X$.

Let us now define precisely what a factor of a measure means
and, in particular, the precise form of our problem.
Let $X=S_1^{{\mathbb Z}_+}$ and $Y=S_2^{{\mathbb Z}_+}$ and let
$\pi: X\to Y$ be symbolic map in the sense that
$\pi=\pi_0\times\pi_1\times\dots$, i.e. if $y=\pi(x)$, then
$y_i=\pi_i(x_i)$. (This explains the terms ``single site
renormalization'' or the ``$1$-block'' factor.) Our problem is to find suitable
conditions for a $g$-measure $\mu$ on $X$ to be pushed down by
$\pi$ to a $g$-measure $\tl\mu$ on $Y$ (i.e., the 1-block factor
$\tl\mu = \mu \circ \pi^{-1}$ of $\mu$ is a $g$-measure).

In the literature, the best results are those by
Verbitskiy \cite{verb}, and Wang and Redig \cite{redig}, respectively,
where these authors assume summability of variations of the $g$-function,
which means that 
\begin{equation}\label{sumv}
\sum_{n=1}^\infty \var_n g <\infty,
\end{equation}
where 
$$\var_n g=\sup_{x\sim_n y}|g(x)-g(y)|,$$
where $x\sim_n y$ if $x,y\in X$ coincide in the first $n$
coordinates. In the ergodic theory literature one often imposes
summability of the sequence $\var_n \log g$, but these conditions are
equivalent so long as the $g$-function is regular, that is, if $g>0$.

The results and open questions concerning factors of $g$-measures are closely related 
to sufficient conditions for uniqueness of a $g$-measure. Doeblin and Fortet \cite{doeblin} 
famously showed that uniqueness of a $g$-measure follows from condition \eqref{sumv}.
This condition was weakened by Berbee \cite{berbee87} to the condition
\begin{equation}\label{berbee}
  \sum_{n=1}^\infty e^{-r_1-r_2-\cdots r_n}=\infty,
\end{equation}
where $r_n=\var_n \log g$. This includes the possibility of having
$\var_n \log g=1/n$, but not $\var_n \log g\geq 1/n^{\alpha}$ if
$\alpha<1$. In \cite{johob} the situation was improved considerably for such
sequences, by only requiring square summability of the
variations for uniqueness, that is 
\begin{equation}\label{sqvar}
\sum_n (\var_n \log g)^2<\infty 
\end{equation}for uniqueness. 
Berger {\em et al}.\ \cite{berger} proved that this condition is best possible in the sense that for any $\epsilon>0$, we can find a function $g$ that satisfies
$$\sum_{n=1}^\infty (\var_n \log g)^{2+\epsilon}<\infty,$$
so that there are multiple $g$-measures.
In \cite{jop2}, we improved \eqref{sqvar} as a sufficient condition for uniqueness (and the Bernoulli property) if we
only assume $\var_n \log g=o(1/\sqrt{n})$, as $n\to \infty$. 

Verbitskiy suggested in \cite{verb} that the class of
$g$-measures satisfying \eqref{sqvar} could be a natural class to
consider for being closed under taking $1$-block factors. We will 
present a counterexample to Verbitskiy's conjecture. The question 
remained if there is a broader natural class of $g$-functions to consider 
than those of summable variations to get a factored $g$-measure. 

In fact we prove among other things the following results. 

In {\bf Corollary 3} of {\bf Theorem 2} we prove that under the condition \eqref{berbee}, the factor $\mu \circ \pi^{-1}$ of the unique $g$-measure is also a $g$-measure.
This is an improvement of Verbitskiy's result in \cite{verb}. To prove Theorem 2 we combine Verbitskiy's methods with some estimates from Berbee's paper \cite{berbee89}.

In {\bf Theorem 4} we prove that there exists a $g$-function with $\var_n \log g=O(1/n)$, as $n\to \infty$, and a symbolic map $\pi:X\to Y$, such that the unique (because of condition \eqref{sqvar}) $g$-measure on $X$ has a $1$-block factor $\mu \circ \pi^{-1}$ which is not a $g$-measure.

In view of the fact that $r_n=1/n$ satisfies \eqref{berbee}, we see that Berbee's condition is optimal within constants for the $g$-measure property to hold under taking $1$-block factors.

Theorem 4 provides the counterexample to Verbitskiy's conjecture in \cite{verb} and is a construction of (a unique) $g$-measures when we have multiple Gibbs measures. The $g$-function is constructed from a general potential that admits two Gibbs measures, one of which dominates the other in that it gives a bigger value when integrating a strictly increasing function, and this gives a non-continuous induced $g$-function for a certain $1$-block factor of the original $g$-measure. We now present a brief explanation and context for the construction.

We exhibit two distinct eigen-measures of the adjoint $L^*$ of the transfer operator $L$ that acts on continuous functions $f$ on spaces like $X=S^{{\mathbb Z}_+}$ as
$$Lf(x)=\sum_{y\in T^{-1}x}e^{\phi(y)}f(y).$$
These correspond to probability measures $\nu$ that satisfy $L^*\nu=\lambda \nu$, where $\lambda>0$ is the spectral radius of $L$, and these eigen-measures coincide with one-sided Gibbs measures (see, e.g., \cite{walters}, Corollary 2.7, or the investigation in \cite{lop3}). We can also see that $g$-measures are special cases of such one-sided Gibbs measures, since for a given $g$-function $g$ we can define a transfer operator $L_g$ by
$$L_gf(x)=\sum_{y\in T^{-1}x}g(y) f(y),$$ where we have imposed $\sum_{y\in T^{-1}x}g(y)=1$ for all $x$. Thus the $g$-measures satisfy $L_g^*\mu=\mu$. In the construction of the counterexample we use that in general we do not have a unique Gibbs measure under the condition \eqref{sqvar}. In fact, it is known (\cite{FS},\cite{ACCN}) that there is a phase transition for the one dimensional two-sided sided Ising model with long-range interaction. We use this to obtain a one-sided potential $\phi$ for which $\var_n \phi=O(1/n)$, and such that there are multiple solutions $\nu$ to $L^*\nu=\lambda \nu$. This doesn't follow automatically, since we cannot use Sinai's famous ``lemma'' (\cite{sin}) that essentially says that we can study a two-sided Gibbs measure as a one-sided version if $\sum_n n\var_n \phi<\infty$. 

In {\bf Theorem 1} we prove that a ``one-sided'' version of the Ising model also has a phase transition at a critical inverse temperature at most 8 times the critical temperature of the original two-sided model. This implies that the critical level for obtaining multiple Gibbs measures are the same, i.e., when $\var_n \phi=O(1/n)$ (the variations are defined in an analogous way for two-sided systems). We conjecture that we have the same critical temperature for the one-sided Ising model.  We use this one-sided long-range Ising model
to construct a $g$-function for Theorem 4.
\newline

\noindent
{\bf Acknowledgement}. The authors would like to thank Jeff Steif for stimulating conversations.

\section{Gibbs measures, $g$-measures and the long-range Ising model}

\def\X{\mathit{X}} \def\CP{\mathcal{P}} \def\T{\mathit{T}}
\def\rst#1{_{#1}} \def\tailint#1{{(#1)}} \def\hdint#1{{[0,#1)}}
\def\ti#1{\rst{\tailint{#1}}} \def\he#1{\rst{\hdint{#1}}}

\subsection{Symbolic spaces}

For a measurable space $(X,\CF)$, let $\CM(X,\CF)$ denote the space of
bounded measures. Let $\mc C(X)$ denote the space of continuous
functions on a topological space $X$.  In what follows $S$ is a
countable set and $X$ a product set of the form $X=\prod_{i\in S} A_i$
where $A_i$ are finite sets. (A ``symbolic'' space.) We assume $X$ is
equipped with the product topology and the corresponding Borel algebra
$\mc F = \CB(X)$. We say that $X$ is homogeneous if the $A_i$ are all
equal, i.e. if $X=A^S$ for some fixed finite set $A$ of symbols. If
$S=\ZZ$ or $S=\ZZ_+$ we have the left-shift operator $T: X\to X$, by
$(\T x)_n=x_{n+1}$.

For any subset $F$ of $S$ an element $x=(x_s)\in S$ in $X$ can be
represented as $x=x_F \times x_{F^c}$, where
$x_F \in X_F := \prod_{s\in F} A_s$.  The sub sigma-algebra of $\mc F$
generated by $x_F$ is $\mc F_F$.  We write $x\sim_F y$ if $y_i=x_i$
for all $i\in F$. In the following, we use $\Lam$ to refer to a finite
set $\Lam\subset S$ and $\Lamc$ means the complement of $\Lam$.  We
write $\Lam_n\to S$ for taking limits with respect to an increasing
sequence $\{\Lam_n\}$ of finite sets such that, eventually,
$F\subset \Lam_n$ for any finite set $F$.

We will mostly work with $S$ being the set $\ZZ_+=\{0,1,2,\dots\}$ of
positive integers with finite sets $\Lam$ of the form $\Lam_n=[0,n)$.
In this case $x\sim_n y$ means $x \sim_{[0,n)} y$.  We use $(n)$ for
the complement $\Lamc=\ol{[0,n)}=[n, \infty)$ of $\Lam=[0,n)$. We prefer to write
$x\ti n$ for the tail sequence $(x_n,x_{n+1},\dots)$.

For a measure $\mu\in\CM(X)$ and a subset $F\subset S$, let
$\mu_F = \mu\circ(x_F)^{-1} \in \CM(X_F,\CF_F)$ denote the marginal
distribution of $x_F$.

\subsubsection{Stochastic dominance}
Here we will give some definitions that are used in the construction of our 
counterexample (Theorem 4). We assume that the symbolic space $X$ is 
partially ordered: $x\leq y$
meaning that $x_i\leq y_i$ for all $i\in S$, where we assume the
$x_i\in A_i$ are integers.  This order induces a partial order
$\preceq$ on the space $\CM(X)$ of probability measures on $X$: For
two probability measures $\mu,\mu' \in\CM(X)$ on $X$, we say that
$\mu'$ \emph{stochastically} \emph{dominates} $\mu$ if
$\mu'(f) \geq \mu(f)$ for every increasing function $f:X\to\RR$. We
write the stochastic dominance relation $\mu\preceq \mu'$. Strict
dominance, written $\mu\prec \mu'$, means that $\mu(f)<\mu'(f)$ for
every strictly increasing function $f$.

An equivalent formulation is that $\mu\preceq\mu'$ whenever we can
define $x'\in X$ with distribution $\mu'$ and $x\in X$ with
distribution $\mu$ on a common probability space, i.e., we \emph{couple}
$\mu$ and $\mu'$, in such manner that $\P(x \leq x') = 1$. Strict
dominance means that the coupling allows $\P(x<x')=1$.

\subsubsection{Potentials and one-point potentials}
Let $X$ be a symbolic space. In this paper, we will refer to a \emph{potential} 
$\phi$ as an equivalence class of pointwise limits of functions on $X$ with the equivalence 
relation that whenever $x$ and $y$
coincide outside a finite set, the difference
$$\phi(x)-\phi(y):=\lim_{\alpha} (\phi_{\alpha}(x)-\phi_{\alpha}(y))$$  
is well-defined (both with respect to the limit and to the equivalence relation) 
and is finite. 

For a potential $\phi$ and a finite
set $\Lam\subset S$ and an arbitrary but fixed mapping
$\Lam\to K(\Lam)\in\RR$, assigning a constant ``ground potential'' to
each finite set $\Lam$, we can define a function
$$\phi_\Lam(x) = \max_{y} \{\phi(x)-\phi(y) + K(\Lam): y\,\sim_\Lamc\, x\}, $$
with $\phi_{\emptyset}(x)=0$.  
Note that the difference 
$$\phi_\Lam(x)-\phi_\Lam(y)=\phi(x)-\phi(y)$$
for all $x$ and $y$ such that $x\sim_\Lam y$. 

We will mostly assume (except for the random cluster model later) that the potentials are
\emph{continuous} with respect to the product topology, which means that
\begin{equation}\label{contcond}
  \text{the functions $\phi_\Lam(x)$ are all continuous on $X$.}
\end{equation}

For a given sequence $\Lam_n\nearrow S$, $n\geq0$, we can represent
the potential $\phi(x)$ as the limit
$$ \phi(x) = \lim_{\Lam_n\nearrow S} \phi_{\Lam_n}(x). $$
The limit may not exist, but the difference 
$$\phi(x) - \phi(y) = \lim  (\phi_{\Lam_n}(x) - \phi_{\Lam_n}(y)) $$
should exist.  We can also represent $\phi$ as a limit in the sense of the following infinite series
$$ \phi(x) = \sum_{n=0}^\infty \phi_n(x), $$
where $\phi_n(x)=\phi_{\Lam_n} - \phi_{\Lam_{n-1}}$.  The sequence
$\phi_n(x)\in \mc C(X)$ is the \emph{one-point potential} of
$\phi$. (It is the potential of the point
$x_{\Lam_n\setminus \Lam_{n-1}}$. In the case $S=\ZZ_+$ it is the potential of
the point $x_{n-1}$.)

\subsection{Gibbs distributions}
A \emph{Gibbs distribution} $\mu$ on $X$ with potential $\phi$ is a
probability distribution $\mu$ on $X$ such that for any finite set $\Lam$
the conditional probability of $x$ (or equivalently $x_\Lam$) given
$x_{\Lamc}$ is proportional to $\exp(\phi_\Lam(x))$. Such a specification is automatically consistent, since
the ratio
$$\exp(\phi_{\Lam}(x))/ \exp(\phi_{\Lam}(y))=\exp(\phi(x)-\phi(y))$$
is constant for $x\sim_{\Lamc}y$.

That is, we specify that
a version of the conditional probability $\mu(\cdot|\CF_\Lamc)$ satisfies 
\begin{equation}\label{gibbs2}
  \mu(x\mid x_\Lamc) = \rho_{\phi,\Lam}(x\mid x_\Lamc) := 
  \frac{\exp\left(\phi_\Lam(x) \right)}{Z_\Lam(x_\Lamc)},
\end{equation}
where $Z_\Lam=Z_\Lam(\phi)$ denotes the local partition function
$$
Z_\Lam(x_\Lamc) = Z_\Lam(x_\Lamc;\phi)=\sum_{y_\Lam\in X_\Lam}
\exp\left(\phi_\Lam(y_\Lam \times x_\Lamc) \right).
$$
For a given potential $\phi$, we denote the set of corresponding
Gibbs distributions by $\mc G(\phi)$.

A probability measure $\mu$ is $\Lam$-Gibbsian with respect to $\phi$,
$\mu\in\mc G_\Lam(\phi)$, if for all $f\in \mc C(X)$
$$ \mu(f) = \int f(x)\, \rho_{\phi,\Lam}(x\mid x_\Lamc) \d\mu_\Lamc(x_\Lamc). $$
Thus $\mc G(\phi) = \cap_\Lam \mc G_\Lam(\phi)$.

Alternatively, one can define $\mc G(\phi)$ as the set of weak limits
in $\CM(X)$ of consistent sequences of finite support probability measures of the form given
in \qr{gibbs2} with respect to some filtration $\Lam\nearrow S$.  For
a fixed $\xi\in X$ and a filtration $\Lam$, we say that a limit
\begin{equation}\label{gibbs3}
  \mu = \lim_{\Lam\nearrow S} \rho_{\phi,\Lam}(x| \xi_\Lamc)
\end{equation}
corresponds to a Gibbs measure with \emph{boundary condition} $\xi$.

\subsubsection{The case $S=\ZZ_+$} 
In the case $S=\ZZ_+$, we use $\Lam_n=[0,n)$ and the one-point
potential sequence
$$\phi_n(x) = \phi_{[0,n+1)}(x) - \phi_{[0,n)}(x), $$
and the potential representation 
$$ \phi(x) = \sum_{n=1}^\infty \phi_n(x). $$
If 
$$ \phi_n(x) = \phi_n(x_n,x_{n+1},\dots) = \phi_0(\T^n x) $$
then we say the potential sequence is \emph{homogeneous}. 

If, for the sequence of potentials $(\phi_n)$, we have that
$$ \sum_{x_{n}\in A_{n}} \exp(\phi_n(x_n,x_{n+1},\dots)) = 1, \quad \forall n\
\forall x $$ then $(\phi_n)$ is said to be
\emph{normalised}. Equivalently, we have that the local partition
functions $Z_n(x):=Z_{[0,n)}(x) \equiv 1$ for all $n\geq 0$.  

For a normalised and homogeneous sequence of one-point potentials
$(\phi\circ T^n)$ the function $q(x)=e^{\phi_n(x)}$ is referred to as
a \emph{$g$-function} and the corresponding Gibbs measures are
\emph{$g$-measures}. A $g$-measure $\mu$ is always shift-invariant,
i.e.\ $\mu=\mu\circ T^{-1}$.

\subsection{The long-range Ising model}

A relevant example of a potential and a Gibbs measure is the
\emph{long-range (ferromagnetic) Ising model} on $$U=\{-1,1\}^S,$$
where we compare the two-sided case $S=\ZZ$ with the one-sided case
$S=\ZZ_+$.  Let $S^{(2)}$ denote the set of unordered pairs $ij$ of
elements in $S$.  We refer to $S^{(2)}$ as the complete graph on $S$
and its elements $ij$ as edges.  We usually exclude loops, i.e. the edges of the form $ii$ for $i\in S$.

The long-range Ising model is defined by the potential $\varphi(u)$,
$u=(u_i)\in U$, given by
\begin{equation}
  \varphi(u) = \varphi(\alpha,\beta)(u) := 
  \beta \sum_{ij\in S^{(2)}} \frac{u_iu_j}{|i-j|^\alpha}, 
  \quad \alpha >1, \ \beta\geq 0.
  \label{potalpha}
\end{equation}
The potential $\varphi(\alpha,\beta)$ is not well-defined for
$\alpha\leq 1$.

For $S=\ZZ_+$, we obtain the potential $\varphi$ from the homogeneous
sequence of potentials $(\varphi_0(T^n u))$ where (with $\alpha=2$)
\begin{equation} \label{potone} 
\varphi_0(u) = K+\beta \sum_{j=1}^\infty \frac{u_0u_j}{j^2},
\end{equation}
where we can choose $K$ to be arbitrary.

Let $\mc I=\mc I(\beta,\alpha,S)$ denote the set
$\mc G(\varphi(\alpha,\beta))$ of Gibbs measures for the
Ising-potential above.  We use $\nu_+$ and $\nu_-$ in $\mc I$ to
denote the Gibbs-measures obtained as limits by taking the constant
sequences $\xi=\ol{+1}$ and $\xi=\ol{-1}$ as boundary conditions,
respectively.  In the two-sided case, it is well-known (see
\cite{ACCN}) that for $\alpha$ in the range $(1,2]$ there is a
\emph{critical inverse temperature} $\beta_c=\beta_c(\alpha)$. This
means that for $\beta>\beta_c$ we have the strict stochastic dominance
relation $\nu_+ \succ \nu_-$. Moreover, $\nu_+=\nu_-$ whenever
$\beta<\beta_c$ so that the Gibbs measure is unique. We also have
that $\nu_+\succ\nu_-$, when $\beta=\beta_c$ but we do not need this
result. For $\alpha>2$, we have uniqueness of Gibbs measures.

We use this result to prove that there is a critical inverse temperature in
the one-sided case as well. We need to establish the existence of a
critical temperature in order to construct a counterexample to
Verbitskiy's conjecture later.
\begin{theorem}\label{th1}
  For the one sided case, i.e.\ considering
  $\mc I = \mc I(\alpha,\beta,\ZZ_+)$, $1<\alpha\leq 2$ and
  $\beta\geq 0$, there is a similar critical inverse temperature
  $\beta^+_c$. Moreover, $\beta_c^+(\alpha) \leq 8\beta_c(\alpha)$.
\end{theorem}
The proof is postponed until Section 4.

It should be remarked that the factor $8$ in the bound on $\beta^+_c$
is an artifact of the proof. We conjecture that $\beta^+_c$ is indeed
equal to the two-sided $\beta_c$ for the relevant values of $\alpha$.

\section{An optimal condition for factors of $g$-measures}

\subsection{Uniqueness of Gibbs measures and Berbee's condition}

\subsubsection{Transfer operator}

The study of Gibbs measures in \cite{berbee89} is based on the
analysis of the generalised ``transfer operator'' $L=(L_n)$ and its
dual $L^* = (L^*_n)$: For a given sequence of potentials $(\phi_n)$,
let $\CM\ti n=\CM(\X\ti n,\CF\ti n)$.  We define the transfer operator
$L=(L_n)$ as the system of maps
\begin{equation*}
  \mc C(\X\ti0) \xrightarrow{\ L_0\ } \mc C(\X\ti1) \xrightarrow{\ L_1\ }   
  \cdots  \mc C(\X\ti n) \xrightarrow{\ L_n\ } \cdots 
\end{equation*}
where $L_{n}: \mc C(\X\ti {n})\to \mc C(\X\ti{n+1})$ is given by
$$ 
(Lf)(x\ti{n+1}) = \sum_{x_{n}\in A_{n}} \phi_{n-1}(x_{n}, x\ti {n+1})
f(x_{n}, x\ti {n+1}).
$$
Dually, we obtain the system $L^*$ of maps between measures
\begin{equation*}
  \CM\ti0 \xleftarrow{\ L^*_0\ } \CM\ti1 \xleftarrow{\ L^*_1\ }   
  \cdots \CM\ti n \xleftarrow{\ L^*_n\ } \cdots 
\end{equation*}
where $L_{n+1}^*: \CM\ti{n+1}\to \CM\ti{n}$ is given by ``multiplication
by $\exp(\phi_n(x))$''.

A Gibbs measure $\mu$ on $\X$ corresponds to a projective limit for
the system above: Recall that $\mu\ti n$ denotes the restriction of
$\mu$ to $\CF\ti n$. If
$$ \mu_n := \frac 1{Z_n} \mu\ti n, \quad n\geq 0 $$
then it is readily checked that the sequence
$(\mu_0,\mu_1,\dots)\in\prod_n\CM\ti n$ satisfies
$$ \mu_n = L^* \mu_{n+1}. $$
We write $\mu=\mu_0=(L^*)^n\mu_n$.  It is also clear that if $L^*$ is
defined as multiplication by $e^{\phi_n}$ then any such sequence
$(\mu_0,\mu_1,\dots)$ where $\mu_n=L^*\mu_{n+1}$ gives the Gibbs measure
$\mu_0\in\mc G(\sum_n \phi_n)$.

\subsubsection{Berbee's condition}
The relation between smoothness of the transfer operator and
uniqueness of Gibbs measures is a central object of study. We measure
the smoothness of the sequence of potentials $(\phi_n)$ with uniform
variations $r_k$, $k\geq1$, defined as
\begin{equation}\label{rdef}
  r_k = \sup_n \var_k \phi_n(x\ti n).
\end{equation}
It is shown in \cite{berbee89} that the Gibbs measure
$\mu\in\mc G(\sum_n \phi_n)$ is unique whenever
\begin{equation}\label{berbeecond}
  \sum_{n=1}^\infty e^{-r_1-\cdots-r_n}=\infty. 
\end{equation}
We refer to this as ``Berbee's condition''.

For two bounded measures $\nu,\tl\nu$ on a symbolic space $X$, we
define
$$ \rho_k(\nu,\tl\nu) = \inf_{C\in\CF_k} \frac{\tl\nu(C)}{\nu(C)}. $$
Let $P=(P_{ij})_{i=0,j=0}^{\infty,\infty}$ be the Markov matrix given
by $P_{00}=1$ and $P_{0j}=0$ and for $i>0$
$$ 
P_{ij} = \begin{cases}
  0 & j < i-1 \\
  e^{-r_{j}} & j=i-1 \\
  e^{-r_{j+1}} - e^{-r_{j}} & j\geq i
\end{cases},
$$
and where $r_j=r_j((\phi_n))$ are the variations from \qr{rdef}.  Let
$X=B_0\times B_1 \times \dots$ and $X'=B_1\times B_2\times \dots$ and
consider a transfer operator $L: \mc C(X)\to \mc C(X')$ given by
$$ Lf(x') = \sum_{x_0} f(x_0,x') e^{\phi(x_0,x')}.$$

Let $N\geq 1$. A probability measure $\nu\in\CM(\X)$ is
$[0,N)$-Gibbsian if it is extended by $L$ from the restriction
$\nu\ti N$ on $\X\ti N$ (the ``boundary condition''). That is, if we
have
$$\nu = (L^*)^N \nu_N = L_0^* L_1^* \dots L^*_{N-1} \nu\ti N. $$ 
Berbee shows with a clever induction argument that for any pair of
$[0,N)$-Gibbsian measures $\nu,\tl\nu$ it we have that
\begin{equation}\label{mixrate}
  \rho_k(\nu,\tl\nu) \geq \P(Z_N = 0|Z_0=k)^2,
\end{equation}
where $Z_t$, $t=0,1,\dots$ denotes a Markov chain on state-space
$[0,\infty)$ with transition matrix $P$. Note that $Z_t$ is absorbing
at state $0$ and that Berbee's condition \qr{berbeecond} is equivalent
to stating that $Z_t$ is recurrent and hence that
$$ \lim_{N\to\infty} \P(Z_N = 0|Z_0=k)^2 = 1.  $$
The uniqueness of the Gibbs measure then follows.
 
\subsection{One-factors and Gibbsianity}

\def\Y{Y} \def\y{y}

Let $X = \prod_{n=1}^\infty A_n$ and $Y=\prod_{n=1}^\infty \tl{A}_n$
be two symbolic spaces as defined above.  For our purposes a
\emph{symbolic map} is a map $\pi: \X \to \Y$ obtained from a sequence
of surjective coordinate-wise maps $\{\pi_i: A_i\to \tl A_i\}$ so that
$\pi(x)_i=\pi_i(x_i)$.
\begin{theorem}\label{thm1}
  Assume that $\mu$ is Gibbs measure on $X$ with respect to the
  potential $\sum_n \phi_n$ and that the sequence $(\phi_n)$ of
  one-point potentials satisfies Berbee's condition
  \qr{berbeecond}. For any symbolic map $\pi:\X\to\Y$ the following hold.
    \begin{enumerate}[(i)]
  \item\label{cont1} The conditional probability measure
    $\mu(x|\y)\in\CM(X)$ is a continuous function of $\y=\pi(x)\in\Y$.
  \item\label{gibbs1} If, in addition, the potential sequence
    $(\phi_n)$ is normalised, then the distribution of $\y$,
    $\tl\mu = \mu\circ\pi^{-1}$ is given by the normalised potential
    sequence $(\log \tl p_n)$ where
    \begin{equation}\label{pdef1}
      \tl p_n(\y\ti{n}) = 
      \int \sum_{x_n\in\pi_n^{-1}(y_n)} 
      e^{\phi_n(x_n,x\ti{n+1})}\d\mu(x\ti{n+1}\mid\y\ti{n+1}).
    \end{equation}
  \end{enumerate}
\end{theorem}

\begin{proof}[Proof of Theorem~\ref{thm1}]
  The argument is in many ways similar to that of Verbitskiy in \cite{verb}:
  We note that $\mu(x|\y)$ is a Gibbs measure on the non-homogeneous
  symbolic space
  $$\pi^{-1}(\y)=\pi_0^{-1}(\y_0)\times \pi_1^{-1}(\y_1) \times \cdots, $$
  with respect to $(\phi_n)$ restricted to
  $\pi^{-1}(\y)$: For any $x\in\pi^{-1}(\y)$, the probability
  $\mu(x\mid\y,x\ti n)$ differ from $\mu(x\mid x\ti n)$ by the factor
  $1/\mu(\pi^{-1}(\y)\mid x\ti n)>0$ and hence it is proportional to
  the product $\exp(\sum_{k=0}^{n-1}\phi_k(x\ti k))$, since this holds for
  $\mu(x\mid x\ti n)$ by assumption.

  In order to prove continuity of the map
  $\y\mapsto \mu(x|\y)$, we use the explicit mixing rate in
  \qr{mixrate}. If $\y$ and $\y'$ are two different element in $\Y$
  such that $\y\sim_N \y'$ then $\mu(x|\y)$ and $\mu(x|\y')$ are
  $[0,N)$-Gibbs measures on the space
  $$ \pi_0^{-1}(\y_1) \times \cdots \times \pi_{N-1}^{-1}(\y_{N-1})
  \times A_N \times A_{N+1}\times\cdots $$ with respect to $(\phi_n)$. Hence,
  by \qr{mixrate}, we have
  $$ |\log \mu([x]_k | \y) - \log \mu([x]_k | \y')| \leq |\log
  \P(Z_N=0|Z_0=k)| $$ which tends to zero as $N\to\infty$ by Berbee's
  condition.
  
  To show \qr{gibbs1}, we note that $(\tl\phi_n)$ is
  clearly normalised, since
  $$ \sum_{y_n\in\tl{A}_n} \tl p_n(y_n,y\ti{n+1}) = 
  \int \underbrace{\left(\sum_{x_n} p_n(x_n,x\ti{n+1})\right)}_{=1}
  \d\mu(x\ti{n+1}|y\ti{n+1}) = 1,
  $$
  where $p_n(x)=e^{\phi_n(x)}$ and $\tl p_n(x) = e^{\tl\phi_n(x)}$.
  That $\tl p_n$ is continuous follows from the continuity of
  $y\mapsto\mu(\cdot| \y\ti{n+1})$.

  That the distribution of $\y$, i.e.\ $\tl\mu=\mu\circ\pi^{-1}$, is a
  Gibbs measure with the normalised potential $(\log \tl p_n)$ follows
  if $\tl\mu(\y_n|\y\ti{n+1})=\tl p_n(\y_n,\y\ti{n+1})$. But this is
  immediate from the definition of $\tl p$, since
  $$ \tl\mu(\y_n|\y\ti{n+1}) 
  = \int \sum_{x_n\in\pi_n^{-1}(y_n)} p_n(x_n,x\ti{n+1})
  \d\mu(x\ti{n+1}|\y\ti{n+1}) = \tl p_n(y_n,y\ti {n+1}).
  $$
\end{proof}

A symbolic map $\pi:X\to Y$ between two homogeneous spaces $X=A^S$ and
$Y=\tl A^S$ is
homogeneous if it has the form $y_i=\pi(x_i)$ where $\pi: A\to \tl A$
is a fixed surjective map between finite sets. From the explicit form
\qr{pdef1} of the induced potential sequence, it is clear that
homogeneity is preserved if the symbolic map $\pi$ is homogeneous. We
obtain, as a corollary, the result by Verbitskiy in \cite{verb} under
weaker assumptions.
\begin{corollary}
  Assume that $\mu$ is a $g$-measure where the $g$-function satisfies
  Berbee's condition. If $\pi:\X \to \Y$ is a homogeneous factor then
  $\tl\mu=\mu\circ\pi^{-1}$ is a $g$-measure.
\end{corollary}

\subsection{The counterexample to a conjecture by Verbitskiy}
In \cite{verb} (p.\ 328), Verbitskiy argues that it would be natural
to conjecture that for any homogeneous symbolic map $\pi$ and any
$g$-measure $\mu$ with respect to a $g$-function having square
summable variations the measure $\mu\circ\pi^{-1}$ is a
$g$-measure. The condition of square summability variations, i.e.\
$\sum_n (\var_n \, g)^2<\infty$, is a condition that in \cite{johob}
was used to prove uniqueness of the $g$-measure.  Verbitskiy
speculates that the condition of square summability could be closed
under taking 1-block factors, perhaps by adapting arguments
from Fan and Pollicott \cite{fan}.

We show that this is not the case using the following counterexample,
where we have a square summable $g$-function. In fact, the
$g$-function satisfies $\var_n g = \Ordo{1/n}$ and the sequence of
variations is thus only a factor away from satisfying Berbee's
condition.  The construction also connects our investigation with the
principle proposed by van Enter {\em et al}.\ in \cite{enter} and
discussed in \cite{verb} that non-Gibbsianity of factors is linked to
the presence of a ``hidden phase transition''.

\begin{theorem}
  There exists a $g$-function $g$ with
  $\var_n g = \Ordo{1/n}$ and with unique $g$-measure $\mu$, and a
  symbolic map $\pi: X\to Y$, so that the corresponding 1-block factor
  $\mu\circ \pi^{-1}$ is not a $g$-measure.
\end{theorem}

\begin{proof}
  Let $(\varphi(T^n u))_n$ be the homogeneous potential of the one-sided
  long-range Ising model on $U=\{-1,+1\}^{\ZZ_+}$ defined by
  \qr{potone}.  We can choose the constant $K$ in \qr{potone}, so that
  \begin{equation}\label{sumlessthanone}
    q(\pm1,u) < 1/2, \quad \text{for all $u\in U$,} 
  \end{equation}
  where $q(u)=\exp(\varphi_0(u))$.  Let
  $\nu^\pm_N\in\mc G_{[0,n)}(\varphi)$ denote the $[0,N)$-Gibbsian
  measure on $U$ obtained from the boundary condition
  $\xi = \ol{\pm 1}$. By construction $\nu^{\pm}_N \to \nu^{\pm}$ and
  we choose $\beta>\beta^+_c$ from Theorem~\ref{th1}, so that
  $\nu^+ \succ \nu^-$.

Consider the homogenous space $X=A^{\ZZ_+}$ on four symbols
$$ A=\{+1, -1, +\tilde 1, -\tilde 1\}  $$
and the symbolic map $\alpha: X \to U$ defined by
$\alpha(+1)=\alpha(+\tl1)=+1$ and $\alpha(-1)=\alpha(-\tl1)=-1$.  We
define a $g$-function $g(x)$ from $q(u)$ by setting 
$$
g(x) =
\begin{cases}
  q(\alpha(x)),  &\mbox{if } x_0 = \pm 1, \\
  \frac{1}{2}-q(\alpha(x)), &\mbox{if } x_0 = \pm \tilde 1.
\end{cases}
$$
It is obvious that $g$ is a $g$-function.

A simple estimate shows that the log-variations of $q$ in \qr{potone}
satisfy
\begin{equation*}\label{varon}
  \var_n \log q \leq \beta \frac 2n, 
\end{equation*}
and thus the variations of the $g$-function $g$ satisfies
$\var_n \log g = \Ordo{1/n}$.  By \cite{johob} (or \cite{jop2}), we
have a unique $g$-measure $\mu$ on $X$.

Let $Y$ be the symbolic space $Y=B^{\ZZ_+}$ on three symbols
$B=\{0, +\tl 1, -\tl 1\}$ and consider the the shift-invariant factor
$\pi: X \to Y$ defined by $\pi(\pm 1)=0$ and
$\pi(\pm \tl 1)=\pm \tl1$.  Let $\tl\mu=\mu\circ\pi^{-1}$ be the
distribution of $y\in Y$ and let
$$ 
\tl g(y)=\tl\mu(y_0|y\ti1) = \mu\left(\pi(x_0)=y_0\mid \pi(x\ti1) =
  y\ti1\right).
$$  

We claim that the induced $g$-function $\tl g(y)$ for the factor on
$Y$ is discontinuous at $y=\ol0\in Y$, where $\ol0=(0,0,0,\dots)$.

Let $\ol0_N^{\pm} \in Y$ be defined by $y_i=0$ for
$i=0,1,\dots,N-1$ and $y_i=\pm \tl1$ for $i\geq N$.  Since
$\alpha(\pi^{-1}(\ol0_N^+))=\{+1,-1\}^{[0,N)}\times\ol{+1}$, and
since $\alpha(\pi^{-1}(\ol0_N^-))=\{+1,-1\}^{[0,N)}\times\ol{-1}$,
it follows that
  $$ 
  \tl g(+\tl1,\ol0_N^\pm) = \tl \mu({y_0=+\tl1} \mid y\ti1=\ol0_N^\pm) =
  \mu({x_0=+\tl1} \mid \pi(x\ti1) =y\ti1=\ol0_N^\pm)
  $$
  $$
  =\frac 12 - \int q(+1,u) \d\nu_{N}^\pm(u)
  \to
  \frac 12 - \int q(+1,u) \d\nu^\pm(u),
  $$
  as $N\to\infty$.

  Thus, since $u\mapsto q(+1,u)$ is a strictly increasing function on
  $U$, we obtain that
  $$ \lim_{N\to\infty} \tl g(+\tl1,\ol0_N^+) 
  \not= \lim_{N\to\infty} \tl g(+\tl1,\ol0_N^-),
  $$
  which shows the discontinuity of $\tl g$ since $\ol0_N^\pm\to \ol0$
  as $N\to\infty$.
\end{proof}

\section{The one-sided long-range Ising model}

In order to prove Theorem~\ref{th1}, we work with the random cluster
model (as in Aizenman {\em et al} \cite{ACCN}) instead of working with the Ising
model directly.

\subsection{The random cluster model}
A random cluster model on $S$ is a certain type $\mc R(p,q,S)$ of
distribution of a random subgraph $t$ of $S^{(2)}$. We consider
$t=(t_{ij})$ as an element of $T(S):=\{0,1\}^{S^{(2)}}$ and obtain
$\mc R=\mc R(p,q,S)$ as the class of Gibbs measures to the potential
on $T(S)$ given by
$$ 
\log q \cdot c(t) + \sum_{ij} \log(1-p_{ij})(1-t_{ij}) + \log p_{ij}
t_{ij}.
$$ 
Here $c(t)$ denotes the number of connected components (clusters) in
the graph $t$, which readily can be defined as a potential, 
although it is not necessarily continuous.  The random
cluster model has two parameters: The edge-probability
$p: S^{(2)}\to [0,1]$, $ij\mapsto p_{ij}$, and the parameter $q$ which
is a number $q\geq 1$. Note that if $S$ is finite, the distribution of
the random graph $t\sim \psi\in\mc R(p,q,S)$ is a probability
proportional to
$$
q^{c(t)} \cdot \prod_{t_{ij}=1} p_{ij} \cdot \prod_{t_{ij}=0}
(1-p_{ij}).
$$
If $q=1$ then we obtain the standard Bernoulli random graph
distribution on $S$.

We obtain the (free boundary) random cluster distribution
$\psi=\psi(p,q,S) \\ \in\mc R(p,q,S)$ as the limit having fixed boundary
$\xi=\ol 0$ with respect to the sequence $\Lam^{(2)}\nearrow S^{(2)}$
for $\Lam\nearrow S$. (The so-called wired distribution $\psi^w$ is
obtained by taking $\xi=\ol 1$ and
$\Lam_n^{(2) } = \ol{(\Lamc)^{(2)}}\nearrow S^{(2)}$.) By the
monotonicity in $p$, see below, one can deduce that the free boundary
limit $\psi$ is well defined. In all cases of interest in this paper,
$\psi$ is actually the unique random cluster distribution of type
$\mc R(p,q,S)$.

The following three stochastic dominance relations for random cluster
models are well-known (see \cite{ACCN}).  Firstly, the random cluster
distribution $\psi(p,q,S)$ increases with $p$, i.e.\
\begin{equation}\label{incrp}
  p\leq p' \implies \psi(p,q,S) \preceq \psi(p',q,S).
\end{equation}
It decreases in $q$, so that
\begin{equation}\label{incrq}
  q\leq q' \implies \psi(p,q,S) \succeq \psi(p,q',S).
\end{equation}
Finally, we can compare a random cluster distribution with the
corresponding Bernoulli distribution. That is,
\begin{equation}\label{incr2p}
  \psi(p,q,S) \succeq \psi \left(\frac{p}{p+(1-p)q},1,S\right),
\end{equation}
where
$$ \left(\frac{p}{p+(1-p)q}\right)_{ij} = \frac{p_{ij}}{p_{ij}+(1-p_{ij})q}. $$

\subsubsection{The random cluster model and the Ising model coupled}

The long range Ising model $\mc I = \mc I(\beta,\alpha,S)$, $S=\ZZ$ or
$S=\ZZ_+$, can be constructed from the random cluster model. In fact,
we may couple $\nu^+,\nu^-\in \mc I$ using a random cluster
distribution $\psi(\rho,2,S)$ where the edge-probability
$\rho=\rho(\beta,\alpha)$ is given by
\begin{equation}
  \rho_{ij} = 1-\exp\left( -\frac{\beta}{|i-j|^\alpha} \right). \label{rhodef}
\end{equation}
That is, the probability of non-occurrence of the edge $ij$ is
$\exp(-\beta/|i-j|^\alpha)$.

We can (\cite{ACCN}) construct a spin-configuration $u^\pm\in U$
distributed according to $\nu^\pm$ as follows: Choose a graph
$t\sim\psi$ on vertex-set $S$ according to the distribution
$\psi=\psi(\rho,2,S)$ and assign each \emph{infinite} cluster in $X$
the fixed spin-value $\pm1$. For each finite cluster in $t$ a
spin-value in $\{-1,+1\}$ is chosen independently and uniformly at
random.  Then the spin-configuration $u^\pm = u^\pm(t)$ is defined by
setting $u_i$ equal to the spin of the cluster containing $i$. The
spin-configuration $u^-$ is thus equal to $u^+$ except that for $i$
belonging to the infinite cluster the spin $u_i$ is changed from $+1$
to $-1$. It is hence clear that
$$\nu^+ \succ \nu^- \quad \text{precisely if $\psi(A_\infty)=1$} $$
where $A_\infty$ is the event
$$ A_\infty = \text{``$t$ has an infinite cluster''.}  $$
Note that $A_\infty$ is a tail event and that satisfies a zero-one
law, i.e.\ the probability $\psi(A_\infty)$ is either zero or one.

\subsection{Proof of Theorem~\ref{th1}}
Let $\alpha$ be fixed where $1<\alpha\leq2$. Let $\psi(\beta,q,S)$
denote the random cluster distribution $\psi(\rho,q)$, with
$\rho=\rho(\alpha,\beta,S)$ as in \qr{rhodef} above. Note that $\rho$
is increasing as a function of $\beta$.  As stated above, it is known
(see \cite{ACCN}) that if $t$ has distribution
$\psi=\psi(\beta,2,\ZZ)$ then $\psi(A_\infty)=1$, precisely when
$\beta\in[\beta_c(\alpha),\infty)$.

We assume $\beta \geq \beta_c$ and we shall prove that for
$\psi=\psi(8\beta,2,\ZZ_+)$, we have $\psi(A_\infty) = 1$. Since
$A_\infty$ is an increasing event, it is enough to establish a
stochastic dominance $\psi(8\beta,2,\ZZ_+)\succeq \tl\psi$ for some
distribution $\tl\psi$ on $T(\ZZ_+)$, where $\tl\psi(A_\infty)=1$.

Let $F:T(\ZZ)\to T(\ZZ_+)$ be defined by
\begin{equation}\label{Fdef}
  F(t)_{ab} = 1 - \prod_{|i|=a,|j|=b} (1-t_{ij}),\quad t\in T(\ZZ).
\end{equation}
Then $F$ corresponds to the graph homomorphism (loops ``silently
removed'') induced by the vertex map $i\mapsto |i|$. Thus for any pair
$i,j$ of vertices connected by a path in $t$, the images $|i|$ and
$|j|$ under $F$ remain connected in $F(t)$. It is therefore clear that
if $t$ has an infinite cluster then $F(t)$ has an infinite cluster.
Hence, if we construct $\tl\psi$ as the push-forward
$\tl\psi=\psi(\beta,1,\ZZ)\circ F^{-1}$ of the long-range Bernoulli
random graph on $\ZZ$ then $\tl\psi(A_\infty)=1$ whenever
$\psi(\beta,1,\ZZ)(A_\infty)=1$. Since
$\psi(\beta,1,\ZZ)\succeq\psi(\beta,2,\ZZ)$ and since
$\beta\geq\beta_c$, we can deduce that $\tl\psi(A_\infty)=1$. 
By independence, it follows from \qr{Fdef} that $\tl\psi$ is a
Bernoulli random graph distribution $\tl\psi=\psi(\gamma,1)$ with
edge-probability $\gamma$
\begin{equation}\label{gammadef}
  \gamma_{ij} = 1- \prod_{|i'|=i,|j'|=j} \exp{-\frac{\beta}{|i'-j'|^\alpha}}
  = 1-\exp\left\{-\beta( 2|i-j|^{-\alpha} + 2|i+j|^{-\alpha})\right\},
\end{equation}
for $i,j\geq0$.  From \qr{incr2p} and \qr{incrp}, it follows that
$$\psi'\succeq \psi(4\beta,1,\ZZ_+)$$ 
and we are done since \qr{gammadef} implies that
\begin{equation*}
  \gamma_{ij} \geq 
 1-\exp\left(-4\frac{\beta}{|i-j|^{\alpha}}\right)
\end{equation*}
and thus
$$ \psi(4\beta,1,\ZZ_+) \succeq \tl\psi. $$
This concludes the proof. \qed

\noindent
\newline

\noindent
Anders Johansson, Department of Mathematics, University of G\"avle,
801 76 G\"avle, Sweden. Email-address: ajj@hig.se\newline

\noindent
Anders \"Oberg, Department of Mathematics, Uppsala University, P.O.\
Box 480, 751 06 Uppsala, Sweden. E-mail-address:
anders@math.uu.se\newline

\noindent
Mark Pollicott, Mathematics Institute, University of Warwick,
Coventry, CV4 7AL, UK. Email-address: mpollic@maths.warwick.ac.uk
\end{document}